\documentclass[12pt,a4paper]{amsart}
\usepackage{latexsym,amsfonts,amsthm,amsmath,mathrsfs,amssymb}
\usepackage[british]{babel}
\usepackage[dvips]{color}
\usepackage[utf8]{inputenc}
\usepackage[T1]{fontenc}
\usepackage[dvips,final]{graphics}
\usepackage{xcolor}
\usepackage{mathrsfs}
\usepackage{mathtools} 
\usepackage{breqn}
\usepackage{epstopdf}
\usepackage{verbatim}
\usepackage{amscd}
\usepackage{hyperref}
\usepackage[english,capitalize]{cleveref}
\usepackage{aliascnt}

\usepackage{todonotes}

\makeatletter
\def\newaliasedtheorem#1[#2]#3{
	\newaliascnt{#1@alt}{#2}
	\newtheorem{#1}[#1@alt]{#3}
	\expandafter\newcommand\csname #1@altname\endcsname{#3}
}
\makeatother

\theoremstyle{plain}

\newtheorem{theorem}{Theorem}[section]
\newaliasedtheorem{prop}[theorem]{Proposition}
\newaliasedtheorem{lem}[theorem]{Lemma}
\newaliasedtheorem{coroll}[theorem]{Corollary}

\theoremstyle{definition}
\newaliasedtheorem{defi}[theorem]{Definition}
\newaliasedtheorem{quest}[theorem]{Question}
\newaliasedtheorem{fact}[theorem]{Fact}

\theoremstyle{remark}
\newaliasedtheorem{rem}[theorem]{Remark}
\newaliasedtheorem{exa}[theorem]{Example}

\newcommand{\R}{\mathbb{R}}

\newcommand{\F}{\mathbb{F}}
\newcommand{\PP}{\mathbb{P}}

\newcommand{\C}{\mathbb{C}}
\newcommand{\E}{\mathbb{E}}

\newcommand{\ee}{\mathcal{E}}

\let\altphi\phi
\let\phi\varphi
\let\varphi\altphi
\let\altphi\undefined

\newcommand{\D}{\mathscr{D}}

\newcommand{\average}{{\mathchoice {\kern1ex\vcenter{\hrule height.4pt
width 6pt
depth0pt} \kern-9.7pt} {\kern1ex\vcenter{\hrule height.4pt width 4.3pt
depth0pt}
\kern-7pt} {} {} }}

\address{\textsc{Daniela Di Donato}: 
Universit\'a di Pavia, Dipartimento di Matematica, Via Adolfo Ferrata, 5, 27100 Pavia, Italy}
\email{daniela.didonato@unipv.it}

\thanks{
	D.D.D. is supported by PRIN 2022 "\emph{Inverse problems in PDE: theoretical and numerical analysis}" n. 2022B32J5C funded by MUR, Italy, and by the European Union – Next Generation EU}

\title{Hamilton-Jacobi equations involving a Caputo time-fractional derivative}

\date{\today}

\author{ Daniela Di Donato }

\begin{document}

\begin{abstract}
 We prove a representation formula of intrinsic Hopf-Lax type for subsolutions to Hamilton-Jacobi equations involving a Caputo time-fractional derivative.

		\end{abstract}



\maketitle 

\section{Introduction} 	
Consider the Cauchy problem for Hamilton–Jacobi equations 
\begin{equation}\label{cauchypb8}
\left\{
\begin{array}{l}
\partial _{(0,t]}^\beta v  +H( Dv)=0, \qquad \mbox{in } \R^\kappa \times (0,\infty ), \\ 
v (y,0)=h(y), \qquad \,\,\, \qquad \mbox{on } \R^\kappa \times \{t=0\}, \\ 
\end{array}
\right.
\end{equation}
where $h :\R^\kappa \to \R$ is Lipschitz continuous and bounded, $H:\R^\kappa \to \R$ is convex and such that $\lim_{|p|\to \infty} \frac{H(p)}{|p|}=+\infty $ and 
\begin{equation}\label{equationBETA}
    \partial _{(0,t]}^\beta v (x,t)=\frac{1}{\Gamma (1-\beta )} \int_0^t \frac{\partial _\tau v(x,\tau )}{(t-\tau )^\beta }\, d\tau
\end{equation}
is the Caputo time-fractional derivative of order $\beta \in (0,1)$ of $v:\R^\kappa \times \R^+ \to \R.$ Since the Caputo fractional derivative is defined in an integral form, it is a non-local operator. It has 'memory'  property, which means that a present state depends on past states. Derivatives and integrals to non integer orders are part of the fractional calculus and it represents a powerful tool in applied mathematics to study a myriad of problems from different fields of science and engineering, with many break-through results found in mathematical physics, finance, hydrology, biophysics, thermodynamics, control theory, statistical mechanics, astrophysics, cosmology and bioengineering (see \cite{P99, KST06}). A theory of viscosity solutions for a general class of Hamilton-Jacobi equations with Caputo time-fractional derivative have been recently developed in \cite{GN17A, TY17}.

In \cite{CMI19}, the authors show that the value function
\begin{equation*}
    v_\beta (x,t)= \E _{x,t} \left[ \min_{y\in \R^\kappa } \left( E_tL\left( \frac{x-y}{E_t}\right) +g(y)\right)\right]
\end{equation*}
is a.e. solution of the problem \eqref{cauchypb8}. Here, $L:\R^\kappa \to \R$ is the Legendre transform of $H$ defined by $L(q)=\sup_p\{pq - H(p)\}$ and  $E_t$ is a continuous, non decreasing stochastic process defined as the inverse of a $\beta$-stable subordinator $D_t,$ i.e., $E_t:=\inf\{\tau >0\,:\, D_\tau >t\}$ for $t\geq 0.$ The process $X(s)=Y(E_s)$ solves the stochastic differential equation
\begin{equation}\label{cauchypb1.7}
\left\{
\begin{array}{l}
dX(s)=\alpha (s) dE_s, \\ 
X(0)=\xi , \\ 
\end{array}
\right.
\end{equation}
where $Y$ is a $C^1$ function that satisfy the constraint $Y(t) =x, \alpha (s)=\dot Y(E_s)$ and $\xi=Y(0).$

Aim of this paper is to prove a representation formula of intrinsic Hopf-Lax type for subsolutions to Hamilton-Jacobi equation involving a Caputo time-fractional derivative. More precisely, our main objective is to introduce and to analyze   here  the map $u:Y \times \R^+ \to \R$ with $Y\subset \R^\kappa$ compact defined as 
\begin{equation*}
u(x,t) =  \E _{x,t} \left[ \inf_{z\in Y} \left\{ E_t L\left( \frac{f(x)-\pi^{-1}(z)}{E_t}\right) + g(z) \right\}\right],
\end{equation*}
where $f=(f_1, \dots , f_\kappa):Y\to \R^\kappa$ is a continuous section of a quotient map $\pi:\R ^\kappa \to Y$ (see Definition \ref{def_ILS}), $  
g:= \max _{j=1,\dots , \kappa } f_j $ and $L:\R^\kappa \to \R $ is a suitable continuous map. In additional, $f(y)- \pi^{-1}(z) =f(y)-a\in \R^\kappa$ such that $d(f(y), \pi^{-1}(z))=d(f(y),a)$ and,  as above, $E_t$ is a continuous, non decreasing stochastic process defined as the inverse of a $\beta$-stable subordinator $D_t.$ In 
 Theorem \ref{Teorema3.2},  we prove that 
\begin{equation*}
    \partial ^\beta _{(0,t]}u (x,t) + H_1(Du (x,t)) \leq 0,
\end{equation*} where $ \partial ^\beta _{(0,t]}u$ is the Caputo time-fractional derivative of order $\beta \in (0,1)$ of $u$ and 
\begin{equation*}
   H_1(Du (x,t)) := \max_{ q \in Y}  \{Du (x,t) \cdot q -C \sqrt{L}(f(q)) \},
\end{equation*} 
 for some fixed $C>0.$ 

\section{Notation and preliminary results}\label{Intrinsically   Lipschitz  sections.21}
\subsection{Intrinsically   Lipschitz  sections}
 In \cite{DDLD21}, we give the following notion. 
\begin{defi}\label{def_ILS} 
Let $(X,d)$ be a metric space, $Y$ be a topological space and $\pi:X\to Y$ be a quotient map, i.e.,  it is continuous, open, and surjective. We say that a map $\phi :Y \to X$ is a section of $\pi $ if 
\begin{equation*}\label{equation1}
\pi \circ \phi =\mbox{id}_Y.
\end{equation*}
Moreover, we say that a map $\phi:Y\to X$ is an {\em intrinsically Lipschitz section of $\pi$ with constant $\ell$},  with $\ell \in[1,\infty)$, if in addition
\begin{equation*}\label{equationFINITA}
d(\phi (y_1), \phi (y_2)) \leq \ell d(\phi (y_1), \pi ^{-1} (y_2)), \quad \mbox{for all } y_1, y_2 \in Y.
\end{equation*}
Here $d$ denotes the distance on $X$, and, as usual, for a subset $A\subset X$ and a point $x\in X$, we have
$d(x,A):=\inf\{d(x,a):a\in A\}$.
\end{defi}

We notice that if $Y$ is compact, we get that \begin{equation}\label{costuniversale}
K:= \sup _{y_1,y_2 \in Y}  d(\phi (y_1), \pi ^{-1} (y_2)) <+ \infty .
\end{equation}

We underline that, in the case  $ \pi$ is a Lipschitz quotient or submetry \cite{MR1736929, Berestovski},  being intrinsically Lipschitz  is equivalent to biLipschitz embedding, see Proposition 2.4 in \cite{DDLD21}. However, we recall some examples of continuous sections and intrinsically Lipschitz sections.

  \begin{enumerate}
   \item Let the general linear group $X=GL(n,\R)$ or $X=GL(n,\C)$ of degree $n$ which is the set of $n\times n$ invertible matrices, together with the operation of ordinary matrix multiplication. We consider $Y=\R^*= GL(n,\R)/SL(n,\R)$ or $Y=\C^*= GL(n,\C)/SL(n,\C)$ where the special linear group $SL(n,\R)$ (or $SL(n,\C)$)  is the subgroup of $GL(n,\R)$ (or $GL(n,\C)$) consisting of matrices with determinant of $1$. Here the linear map $\pi = det : GL(n, \R)\to \R^*$ is a surjective homomorphism where $Ker (\pi) =SL(n,\R).$
\item Let $X=GL(n,\R)$ as above and $Y=GL(n,\R)/O(n,\R)$ where $O(n,\R)$ is the orthogonal group in dimension $n.$ Recall that $Y$ is diffeomorphic to the space of upper-triangular matrices with positive entries on the diagonal, the natural map $\pi: X\to Y$ is linear. 
\item Let $X=\R^{2\kappa }, Y=\R$ endowed with the Euclidean  distance and $\pi:\R^{2\kappa } \to \R$ defined as $\pi ((x_1, \dots ,x_{2\kappa })):= x_1+ \dots +x_{2\kappa }$ for any $(x_1,\dots, x_{2\kappa }) \in \R^{2\kappa }.$  An easy example of sections of $\pi$ is the following one: let $\phi :\R \to \R^{2\kappa }$ given by  $$\phi (y)=(y+ a_1 f_1(y), -a_1 f_1(y), a_2f_2(y), -a_2 f_2(y),\dots , a_{\kappa } f_{\kappa }(y), -a_{\kappa }f_{\kappa}(y)), \quad \forall  y\in \R,$$  where $a_i \in \R$ and $f_i :\R \to \R$ are  continuous maps for any $i=1,\dots , \kappa $. 
\item Regarding examples of intrinsically Lipschitz sections the reader can see \cite[Example 4.58]{SC16}.
\end{enumerate}

\subsection{The subordinator process}
Fixed a filtered probability space $(\Omega , \F ,\F_t , \PP)$, let $\{D_\tau\}_{\tau \geq 0 }$ be a stable subordinator of order $\beta \in (0,1)$, i.e. a one-dimensional, non decreasing L\'evy process whose probability density function (PDF) has Laplace transform equal to $e ^{-\tau s ^\beta }$. The inverse stable process $\{E_t\}_{  t \geq 0 }$, defined as the first passage time of the process $D_\tau $ over the level $t$, i.e.
$$E_t := \inf \{\tau  > 0 : D_\tau > t \},$$
has sample paths which are continuous, non decreasing and such that $E_0=0$, $E_t\to +\infty $ as $t\to +\infty$. It is worthwhile to observe that $E_t$ does not have stationary or independent increments. The process $E_t$ can be used to model systems with two time scales: a deterministic one given by the standard time t, referred to an external observer, and a stochastic one given by $E_t$, internal to the physical process (see \cite{MGZ14, MS04, MS13}).

We recall some basic properties of the process $E_t$ which we will exploit in the following.
\begin{prop}
    For $t >0$, it holds:
    \begin{enumerate}
\item[ (i)]  For any $\lambda >0$ there exists a non negative constant $C(\lambda , \beta)$ such that
        \begin{equation}\label{equ2.1}
           \E [E^\lambda _t]= C(\lambda , \beta )t^ {\lambda \beta},
        \end{equation}
        with $C(\lambda , \beta )= \frac{\Gamma (\lambda +1)}{\Gamma (\lambda \beta + 1)}.$
        \item[ (ii)] The process $E_t$ has PDF
\begin{equation}\label{equationPDF}
    \ee _\beta (s,t)=\frac{t}{\beta }s^{-1-\frac{1}{\beta }} \D _\beta (ts^{-1/\beta }),
\end{equation}
        where  $\D_\beta $ is the PDF of $D_1.$
            \end{enumerate}
\end{prop}

For the proof of the following result we refer to \cite{MS13}.
\begin{prop}
    The function $\ee _\beta (\cdot , t)$ is a weak solution, in sense of distributions, of
    \begin{equation}\label{Prop2.2}
        \partial ^\beta _{(0,t]} \ee _\beta (r,t) = -\partial _r \ee _\beta (r,t) - \frac{t^\beta }{\Gamma (1-\beta )} \delta _0(r), \quad r\in (0,+\infty).
    \end{equation}
\end{prop}

\subsection{The intrinsic Hopf-Lax formula}
Taking inspiration from \cite[Section 3]{E10} and \cite{CMI19}, we consider the intrinsic Hopf-Lax formula, i.e. the map $u_\beta :Y \times \R^+ \to \R$ with $Y\subset \R^\kappa$ compact defined as 
\begin{equation}\label{equationu21.24}
u _\beta (x,t) =  \E _{x,t} \left[ \inf_{z\in Y} \left\{ E_t L\left( \frac{f(x)-\pi^{-1}(z)}{E_t}\right) + g(z) \right\}\right],
\end{equation}
where $f=(f_1, \dots , f_\kappa):Y\to \R^\kappa$ is a continuous section of a quotient map $\pi:\R ^\kappa \to Y$ and \begin{equation*}
g(y):= \max _{j=1,\dots , \kappa } f_j(y), \quad \forall y\in Y.
\end{equation*}
Here, $f(y)- \pi^{-1}(z) =f(y)-a\in \R^\kappa$ such that $d(f(y), \pi^{-1}(z))=d(f(y),a).$ Moreover, from now on we will assume the continuous map $L:\R^\kappa \to \R $ satisfies the following  properties:
\begin{itemize}
\item It holds
\begin{equation}\label{equationPorpL}
\begin{aligned}
 \E _{x,t} \left[  E_t L\left( \frac{f(x)-\pi^{-1}(z)}{E_t}\right) \right] - \E _{x,t} \left[ E_t L\left( \frac{f(y)-\pi^{-1}(z)}{E_t}\right) \right] \leq \E _{x,t} \left[ C  \sqrt{L}\left( \frac{f(x)-f(y)}{E_t}\right)\right],
\end{aligned}
\end{equation} for any $x,y ,z \in Y$ and $t>0$ where  $C=C(K)>0$ where $K>0$ is given by \eqref{costuniversale}.   
 \item If $0<s<t,$  we ask
\begin{equation}\label{equationPorpL.6giugno}
 \E _{x,t} \left[  E_t L\left( \frac{f(x)-\pi^{-1}(z)}{E_t}\right) \right] \leq  \E _{x,t} \left[  E_s L\left( \frac{f(x)-\pi^{-1}(z)}{E_s}\right) \right],
\end{equation} for any $x,z \in Y.$ 
\end{itemize}

An example of a map that satisfies these properties is $L(v)=1/2|v|^2,$ for every $v\in \R^\kappa .$ Indeed, it is easy to check that  $d(f(x),\pi^{-1}(z)) \leq d(f(x),f(y)) + d(f(y),\pi^{-1}(z)), \forall x,y,z \in Y$  and so since $Y$ is compact, then 
\begin{equation*}
d^2(f(x),\pi^{-1}(z)) - d^2(f(y)-\pi^{-1}(z)) \leq  2K d(f(y),f(x)),
\end{equation*}
 for any $x,y, z \in Y,$ where $K>0$ is given by \eqref{costuniversale} and so \eqref{equationPorpL} holds. On the other hand, \eqref{equationPorpL.6giugno} is true because $E_t$ is non decreasing.

Our first result is the following.
\begin{prop}\label{prop2.6}
    The function $u _\beta $ defined as in \eqref{equationu21.24} satisfies the following properties:
    \begin{enumerate}
   \item It holds \begin{equation*}
|u _\beta (x,t) -u _\beta (y,t)| \leq  C \max\left( \sqrt L\left(\frac{f(x)-f(y)}{t}\right), \sqrt L\left(\frac{f(y)-f(x)}{t}\right) \right), \quad \forall x,y \in Y, \, t>0, \end{equation*} where $C>0$ is given by \eqref{equationPorpL}.
\item If $f$ is an intrinsically $\ell$-Lipschitz map, then there exists a constant $c>0$ such that for all $(x,t)\in Y\times (0,+\infty )$
\begin{equation*}
    |u (x,t)-g(x)| \leq c t^\beta .
\end{equation*} Hence, \begin{equation*}
u_\beta =g, \quad \mbox{ on } Y\times \{t=0\}.
\end{equation*}
\item  $u_\beta (x, \cdot )$ is locally $\beta$-H\"older continuous. 
\end{enumerate}

\end{prop}

\begin{rem}
   Proposition \ref{prop2.6} (3) states that  
 $\partial _t u _\beta (x,\cdot ) \in L^1_{loc} (\R^+)$ for every $x\in Y$. Moreover, the function $u$ is absolutely continuous with respect to $t$ and therefore the time-fractional derivative  $\partial _t u _{(0,t]}^\beta (x,t ) \in L^1_{loc} (0,T)$ is defined for a.e. $(x,t)\in Y\times \R^+$ (see for example \cite[Lemma 2.17]{LL18}). On the other hand, thanks to Proposition \ref{prop2.6} (1), we know that if $f$ is differentiable, then $u_\beta (\cdot , t)$ is so too. 
\end{rem}

\begin{proof}[Proof of Proposition \ref{prop2.6}]
$(1).$ Let $x,y \in Y$ and $t>0.$ We choose $z\in Y$ such that 
\begin{equation*}
u_\beta (y,t) = \E _{x,t} \left[  E_t L\left( \frac{f(y)-\pi^{-1}(z)}{E_t}\right) + g(z)\right].
\end{equation*}
and so by \eqref{equationPorpL} we get 
\begin{equation*}
\begin{aligned}
u_\beta (x,t) - u_\beta (y,t ) & = \E _{x,t} \left[  E_t L\left( \frac{f(x)-\pi^{-1}(z)}{E_t}\right) + g(z)\right] - \E _{x,t} \left[  E_t L\left( \frac{f(y)-\pi^{-1}(z)}{E_t}\right) + g(z)\right]\\
& \leq C\sqrt L\left(\frac{f(x)-f(y)}{E_t}\right).
\end{aligned}
\end{equation*}
Hence, interchanging the roles of $x$ and $y$  we have the thesis.

$(2).$  Let $x \in Y$ and $t>0.$ Choosing $z=x$ we obtain
\begin{equation}\label{equation29.3giugno}
u_\beta (x,t) \leq \E _{x,t} [E_tL(0) +g(x)] = C(1, \beta )L(0)t^\beta + g(x), 
\end{equation} where we used \eqref{equ2.1}. 
Moreover,
\begin{equation*}
\begin{aligned}
u_\beta (x,t)-g(x) & = \E _{x,t} \left[ \inf_{z\in Y} \left\{ E_t L\left( \frac{f(x)-\pi^{-1}(z)}{E_t}\right) + g(z) -g(x) \right\}\right] \\
& \geq  \E _{x,t} \left[ \inf_{z\in Y} \left\{ E_t L\left( \frac{f(x)-\pi^{-1}(z)}{E_t}\right) -\ell d(f(x),\pi^{-1}(z)) \right\}\right]\\
& =- \E _{x,t} \left[  E_t \max _{p= \frac{f(x)-\pi^{-1}(z)}{E_t}} \left\{ \ell |p| - L\left( p\right) \right\} \right]\\
& = - \max_{w\in [0, \ell ]}  H(w) \E _{x,t}  \left[  E_t \right]\\
& = - \max_{w\in [0, \ell ]}  H(w) C(1, \beta ) t^\beta ,
\end{aligned}
\end{equation*}  where in the first inequality we used the fact that if $f$  is an intrinsically $\ell$-Lipschitz section of $\pi,$ then $g$ is so too and in the last equality we used \eqref{equ2.1} again. Here $H$ is the Hamiltonian associated by $L$ in our intrinsic context (see \cite[Section 3.3]{D22.88}).

 Finally, putting together the last inequality and \eqref{equation29.3giugno}, it holds
\begin{equation*}
|u(x,t) - g(x)| \leq \max\{ |L(0)| , \max _{\xi \in [0,\ell ]} |H(\xi)|\} C(1, \beta  ) t^\beta  ,
\end{equation*}
which implies that $u=g$ on $Y \times \{ t=0 \},$ as desired.

$(3).$  First, we show that for any $s,t \in \R^+$ such that $s<t$
\begin{equation*}
u_\beta (y,t)-u_\beta (y,s) \leq 0,
\end{equation*}
for every $y\in Y$ and so 
\begin{equation}\label{servira1}
u_\beta (y,t)-u_\beta (y,s) \leq  C(t-s)^\beta .
\end{equation}
  Fix $y \in Y$ and $0<s<t.$ Choosing $z\in Y$ such that 
\begin{equation*}
u_\beta (y,s) = \E _{y,t} \left[  E_s L\left( \frac{f(y)-\pi^{-1}(z)}{E_s}\right) + g(z)\right].
\end{equation*}
By \eqref{equationPorpL.6giugno}, we deduce that
\begin{equation*}
\begin{aligned}
& u_\beta (y,t) - u_\beta (y,s )\\  
& \leq  \E _{y,t} \left[  E_t L\left( \frac{f(y)-\pi^{-1}(z)}{E_t}\right) + g(z)\right] - \E _{y,t} \left[  E_s L\left( \frac{f(y)-\pi^{-1}(z)}{E_s}\right) + g(z)\right] \\
& \leq 0,
\end{aligned}
\end{equation*} as wished. Finally, we prove that for any $s,t \in [h',h'']$ such that $s<t$
\begin{equation*}
u_\beta (y,t)-u_\beta (y,s) \geq -C(t-s)^\beta , 
\end{equation*}
for every $y$ belong to a compact subset of $Y.$ Choose $z\in Y$ such that
\begin{equation*}
u_\beta (y,t) = \E _{y,t} \left[  E_t L\left( \frac{f(y)-\pi^{-1}(z)}{E_t}\right) + g(z)\right].
\end{equation*}
It follows
\begin{equation*}
\begin{aligned}
u_\beta (y,t) - u_\beta (y,s )  & \geq  \E _{y,t} \left[  E_t L\left( \frac{f(y)-\pi^{-1}(z)}{E_t}\right) + g(z)\right] - \E _{y,t} \left[  E_s L\left( \frac{f(y)-\pi^{-1}(z)}{E_s}\right) + g(z)\right]\\
& \geq - \E _{y,t} \left[  (E_t-E_s) L\left( \frac{f(y)-\pi^{-1}(z)}{E_t}\right) \right]\\
& \geq -c \E _{y,t} \left[  (E_t-E_s)\right]\\
& = -c C(1, \beta ) (t-s)^\beta ,
\end{aligned}
\end{equation*} where in the last equality we used \eqref{equ2.1}. 
Now putting together the last inequality and \eqref{servira1}, the proposition is then proved. 
\end{proof}

\section{Subsolution of Hamilton-Jacobi equation involving a Caputo time-fractional derivative}
In this section we present the main result of this paper. Here, we prove that the map defined as in \eqref{equationu21.24} is a subsolution for Hamilton-Jacobi equation involving a Caputo time-fractional derivative defined as in \eqref{equationBETA}. The case when there is the classical derivative $\partial _t$ instead of $ \partial ^\beta _{(0,t]}$ is studied in  \cite{D22Hopf, D22Hopf2}.

Our goal is the following.
\begin{theorem}\label{Teorema3.2}
Let $Y\subset \R^\kappa$ be a compact set, $f:Y \to \R^\kappa$ be a sublinear section of a quotient map $\pi: \R^\kappa \to Y$ and let $u_\beta :Y \times \R^+ \to \R$ the map defined as \eqref{equationu21.24}. Given $(x,t) \in Y\times \R^+$ such that $Du_\beta (x, t)$ exists and $E_t \geq 1$, we have that for any $q\in Y$
\begin{equation*}
    \partial ^\beta _{(0,t]}u_\beta (x,t) +H_q(Du_\beta (x,t)) \leq 0,
\end{equation*} where $ \partial ^\beta _{(0,t]}u$ is the Caputo time-fractional derivative of order $\beta \in (0,1)$ of $u$ defined as in \eqref{equationBETA} and 
\begin{equation*}
   H_q(Du_\beta (x,t)) =  Du_\beta (x,t) \cdot q -C \sqrt{L}(f(q)),
\end{equation*} with $C>0$ given by \eqref{equationPorpL}.
\end{theorem} 

We record some preliminary observations.

    \begin{lem}\label{Lemma2.4}
    For any $y\in Y$ and any $s,t>0$ such that $s<t$ and $E_t \geq 1$ we have
    \begin{equation*}
        u_\beta (x,t)\leq \E _{x,t} \left[ \min_{y\in Y} \left( C (E_t-E_s) \sqrt{L}\left( \frac{f(x)-f(y)}{E_t-E_s}\right) +u_\beta (y,s)\right)\right],
    \end{equation*}  
    where $C>0$ is given by \eqref{equationPorpL}.
\end{lem}

\begin{proof}
    Fix $(x,t) \in \R^+$ and $s<t$ such that $E_t \geq 1.$ For $y\in Y$ let $Z:\Omega \to \R^\kappa$ be a random variable such that
    \begin{equation*}
        u_\beta (y,s)= \E _{x,t} \left[  E_s L\left( \frac{f(y)-\pi^{-1}(Z)}{E_s}\right) + g(Z)\right].
    \end{equation*}
    We deduce that
\begin{equation*}
\begin{aligned}
u(x,t) & \leq   \E _{x,t} \left[  E_t L\left( \frac{f(x)-\pi^{-1}(Z)}{E_t}\right) + g(Z)\right] \\
 &  \leq  \E _{x,t} \left[ C \frac{E_t- E_s}{E_t} \sqrt{L}\left( \frac{f(x)-f(y)}{E_t- E_s}\right)+ \frac{E_s^2}{E_t} L\left( \frac{f(y)-\pi^{-1}(Z)}{E_s}\right) + g(Z)\right] \\
& \leq \E _{x,t} \left[ C (E_t- E_s) \sqrt{L}\left( \frac{f(x)-f(y)}{E_t- E_s}\right)+ u_\beta (y,s)\right],
\end{aligned}
\end{equation*} where in the second inequality we used \eqref{equationPorpL} and in the third one we used the facts that  $E_t \geq 1$ and $E_t\geq E_s$ because $E_t$ is non decreasing.  This inequality is true for every $y\in Y$ and so we get the thesis, as wished.
\end{proof}

\begin{rem}
    Arguing as in Lemma \ref{Lemma2.4}, it is also possible to prove that if $\tau :\Omega \to (0,t)$  is a stopping time such that $E_t \geq 1$ then
    \begin{equation}\label{equation2.8}
        u_\beta (x,t)\leq \E _{x,t} \left[ \min_{y\in Y} \left( C (E_t-E_\tau ) \sqrt{L}\left( \frac{f(x)-f(y)}{E_t-E_\tau }\right) +u_\beta (y,\tau)\right)\right],
    \end{equation} where $C>0$ is given by \eqref{equationPorpL}.
\end{rem}

\begin{proof}[Proof of Theorem \ref{Teorema3.2}]
    Fix $q\in \R^\kappa $ and $h >0$. Consider the control law $\alpha (s) \equiv q$. Then the solution $X(s)$ of \eqref{cauchypb1.7} is given by $X(s)= x - (E_t - E_s)q$. Define the stopping time
\begin{equation*}
    \tau _h =\sup \{ s\in (t-h, t)\,:\, |X(s)-x|=h \}.
\end{equation*}
By \eqref{equation2.8} for $\tau= \tau_h$ and $y=X(\tau _h),$ we have
\begin{equation*}
        u_\beta (x,t)\leq \E _{x,t} \left[ \left( C (E_t-E_{\tau _h} ) \sqrt{L}\left( \frac{f(x)-f(X(\tau _h))}{E_t-E_{\tau _h} }\right) +u_\beta (X(\tau _h),\tau _h)\right)\right].
    \end{equation*}  Since  $f$ is sublinear
\begin{equation}\label{equation3.4}
       \E _{x,t} \left[ u_\beta (x,t) - u_\beta (X(\tau _h),\tau _h)  \right] \leq C \sqrt{L}(f(q))\E _{x,t} \left[   E_t-E_{\tau _h}  \right].
    \end{equation}
    By Ito's formula \cite{K11}, we also have
\begin{equation*}
\begin{aligned}
    \E_{x,t} \left[ u_\beta (x,t) - u_\beta (X(\tau _h),\tau _h)  \right] & = \E_{x,t} \left[ \int_ {\tau _h}^t du_\beta (X(s),s) \, ds \right] \\
&    =  \E_{x,t} \left[ \int_ {\tau _h}^t \partial _s u_\beta (X(s),s) \, ds + \int_ {\tau _h}^t D u_\beta (X(s),s) \, dX(s) \right]\\
&   = \E_{x,t} \left[ \int_ {\tau _h}^t \partial _s u_\beta (X(s),s) \, ds + \int_ {\tau _h}^t D u_\beta (X(s),s)\cdot q \, dE_s \right].\\
    \end{aligned}
\end{equation*}
Therefore, recalling \eqref{equation3.4},
\begin{equation*}
\begin{aligned}
    & \E_{x,t} \left[ \int_ {\tau _h}^t \partial _s u_\beta (X(s),s) \, ds + \int_ {\tau _h}^t D u_\beta (X(s),s)\cdot q  -C \sqrt{L}(f(q))\, dE_s \right] \\
    & = \E_{x,t} \Big[ \int_ {\tau _h}^t \partial _s u_\beta (X(s),s) \, ds \\
    & \quad \qquad + \int _0^\infty \left( \int_ {0}^r D u_\beta (Y(s),D_s)\cdot q  -C \sqrt{L}(f(q))\, ds \right)   (\ee _\beta (r,t) - \ee_\beta (r, \tau _h))dr \Big] \\
    & \leq 0 
    \end{aligned}
\end{equation*}
where $D_s$ is the inverse of $E_t$, i.e. $E_{D_s}=s$ and $\ee _\beta (s,t)$ is the PDF of $E_t$ as in \eqref{equationPDF}. Dividing the previous inequality by $h$ and passing to the limit for $h\to 0 ^+$, by the Dominated Convergence Theorem we get
\begin{equation}\label{eq3.6}
 \partial _t u_\beta (x,t) + \E_{x,t} \left[  \int _0^\infty \left( \int_ {0}^r D u_\beta (Y(s),D_s)\cdot q  -C \sqrt{L}(f(q))\, ds \right)  \partial _r \ee_\beta (r,t)dr  \right] \leq 0.
\end{equation}
Set $\Phi (r)= \int_ {0}^r D u_\beta (Y(s),D_s)\cdot q  -C \sqrt{L}(f(q))\, ds . $ Since $\ee _\beta $ is a solution of \eqref{Prop2.2}, then we have (see \cite[Lemma 4.2]{cdm19})
\begin{equation*}
    \partial _t \ee _\beta (r,t) =-D_{(0,t]}^{1-\beta } (\partial _r \ee _\beta (r,t)) -\delta _0(r)\delta _0(t), \quad r\in (0,+\infty),
\end{equation*}
where $D_{(0,t]}^{1-\beta }$ is the Riemann-Liouville derivative of order $1-\beta $, which is defined for a continuous function $\phi :[0,t] \to \R$ by
\begin{equation*}
    D_{(0,t]}^{1-\beta } \phi (t):= \frac{1}{\Gamma (\beta )} \frac{d}{dt} \int _0^t  \frac{\phi (\tau )}{(t-\tau )^{1-\beta } } \, d\tau .
\end{equation*}
Therefore, noting $\Phi (0)=0$
\begin{equation*}
    \begin{aligned}
        \E_{x,t} \left[  \int _0^\infty \Phi (r)  \partial _r\ee _\beta (r,t)dr \right] & = - \E_{x,t} \left[ \int _0^\infty \Phi (r) D_{(0,t]}^{1-\beta } \partial _r \ee _\beta (r,t)  \right] -\Phi (0)\delta _0(t)\\
        & = - \E_{x,t} \left[D_{(0,t]}^{1-\beta } \left( \int _0^\infty \Phi (r)  \partial _r \ee _\beta (r,t) \right) \right]\\
         & = - \E_{x,t} \left[D_{(0,t]}^{1-\beta } \left( [\Phi (r) \ee_b\eta (r,t)]_0^{+\infty } -  \int _0^\infty \partial _r\Phi (r)  \ee _\beta \eta (r,t) \right) \right].\\
    \end{aligned}
\end{equation*}
Since $\lim_{r\to \infty }\ee _\beta (r,t ) =0$ and $\partial _r \Phi (r)= D u_\beta (Y(r),D_r)\cdot q  -C \sqrt{L}(f(q))$, we deduce that
\begin{equation*}
    \begin{aligned}
        \E_{x,t} \left[  \int _0^\infty \Phi (r)  \partial _r\ee _\beta (r,t)dr \right] & = \E_{x,t} \left[D_{(0,t]}^{1-\beta } \left(  \int _0^\infty  \ee _\beta (r,t) (D u_\beta (Y(r),D_r)\cdot q  -C \sqrt{L}(f(q))) \right) \right]\\
        & = \E_{x,t} \left[D_{(0,t]}^{1-\beta } (D u_\beta (Y(t),D_t)\cdot q  -C \sqrt{L}(f(q))) \right]\\
        & = D_{(0,t]}^{1-\beta }( D u_\beta (x,t)\cdot q  -C \sqrt{L}(f(q)))
    \end{aligned}
\end{equation*}
Replacing the last equality in \eqref{eq3.6}, we get
\begin{equation*}
   \partial_t u_\beta (x,t) + D_{(0,t]}^{1-\beta }( D u_\beta (x,t)\cdot q  -C \sqrt{L}(f(q))) \leq 0 .
\end{equation*}
Applying the fractional integral $I_{(0,t]}^{1-\beta } \phi (t) =\frac{1}{\Gamma (1-\beta )} \int _0^t  \frac{\phi (\tau )}{(t-\tau )^{\beta } } \, d\tau $ to  the previous equation, we finally get the thesis. 
\end{proof}

\subsection{Declarations} There are no conflict of interests/competing interests.

\end{document}